\documentclass[reqno,11pt]{amsart}
\usepackage{amsmath,amsfonts,amssymb,amsthm}
\usepackage{etoolbox}
\usepackage{latexsym}
\usepackage{mathrsfs}

\usepackage[pagebackref=true, colorlinks=true, citecolor=blue, pdfborder={0 0 .1},breaklinks]{hyperref}

\usepackage[initials,nobysame]{amsrefs}

\usepackage{enumitem}

\title[Analyticity of periodic Lei-Lin solutions of Navier-Stokes]
{Existence and analyticity of the Lei-Lin solution of the Navier-Stokes equations on the torus}
\author[D.M. Ambrose]{David M. Ambrose}
\author[M.C. Lopes Filho]{Milton C. Lopes Filho}
\author[H.J. Nussenzveig Lopes]{Helena J. Nussenzveig Lopes}

\newtheorem{lemma}{Lemma}

\newtheorem{theorem}{Theorem}

\newtheorem{definition}{Definition}
\newtheorem{remark}{Remark}

\newcommand{\vertiii}[1]{{\left\vert\kern-0.25ex\left\vert\kern-0.25ex\left\vert #1
    \right\vert\kern-0.25ex\right\vert\kern-0.25ex\right\vert}}

\renewcommand{\hat}{\widehat}
\renewcommand{\tilde}{\widetilde}

\newcommand{\torus}{\mathbb{T}}

\newcommand{\Z}{\mathbb{Z}}
\newcommand{\PP}{\mathbb{P}}
\newcommand{\R}{\mathbb{R}}
\newcommand{\real}{\mathbb{R}}
\newcommand{\dd}{\ \mathrm{d}}

\DeclareMathOperator{\dv}{div}

\begin{document}


\begin{abstract}
Lei and Lin have recently given a proof of a global mild solution of the three-dimensional
Navier-Stokes equations in function spaces based on the Wiener algebra.  An alternative proof
of existence of these solutions was then developed by Bae, and this new proof allowed for an estimate of the radius of analyticity of
the solutions at positive times.  We adapt the Bae proof to prove existence of the Lei-Lin solution in the spatially periodic setting,
finding an improved bound for the radius of analyticity in this case.
\end{abstract}

\maketitle


\section{Introduction}

Lei and Lin have developed a mild solution of the three-dimensional Navier-Stokes equations
which is global in time for small data in spaces related to the Wiener algebra \cite{LeiLin2011}.
Bae subsequently outlined a different proof of this result, with a benefit of the new approach being
that a lower bound for the radius of analyticity of solutions is established in a straightforward way
\cite{baePAMS}.  The results of both these papers are for solutions with domain $\mathbb{R}^{3};$ in the
present work, we prove existence of the corresponding solutions in $\mathbb{T}^{3}$ instead, and
find an improvement in the lower bound for the radius of analyticity in this case.

The lower bound for the radius of analyticity demonstrated in \cite{baePAMS} is proportional to $\sqrt{t};$
the data need not be analytic at time $t=0,$ but the solution then becomes analytic with radius
at least $\sqrt{t}$ at any positive time.  This bound remains valid for all time.  This lower bound
for the radius of analyticity follows the scaling of the linear parabolic term, i.e. the viscous term,
in the Navier-Stokes equations: since it is a second-order term, the radius of analyticity may be
demonstrated to grow at least like $t^{1/2}.$  This bound has been obtained before in other works
such as \cite{baeBiswasTadmor}, \cite{foiasTemamJFA}, \cite{grujicKukavica}.  For related problems such as the
Kuramoto-Sivashinsky equation, the leading-order linear parabolic term is fourth-order, and these
methods yield instead a lower bound on the radius of analyticity like $t^{1/4}$ \cite{ambroseMazzucato},
\cite{grujicKukavica}.  Biswas attained a general result along these lines
on dissipative equations with quadratic nonlinearities in \cite{biswasGeneral}.

A related result is the proof of global existence of vortex sheets by Duchon and Robert \cite{duchonRobert}.  This result
is also on free space, i.e. the domain is $\mathbb{R},$ and the leading-order linear terms are
first-order.  Thus, consistent with the previous results, the radius of analyticity of solutions grows at least like $t.$
In a series of papers, the first author and collaborators have adapted the Duchon-Robert method
to the spatially periodic setting.  In the spatially periodic case, one attains a radius of analyticity growing at least like $t$ regardless of the
order of the linear terms in the evolution equations \cite{ambroseBLMS}, \cite{ambroseBonaMilgrom},
\cite{ambroseMazzucato}.  This linear-in-time growth of the radius of analyticity is an improvement on
long time scales over the fractional-power growth guaranteed by other methods on free space, but
does represent slower growth initially.

In the present work, we combine the two approaches, demonstrating that the lower bound on the radius of analyticity
for the Lei-Lin solution of the three-dimensional Navier-Stokes equations on the torus is initially like $\sqrt{t},$ later improving to $t$.  We mention that Foias and Temam mention that one can prove a lower bound on the radius of analyticity for solutions of the Navier-Stokes equations on the torus growing linearly in time, but their result demonstrates this only for
a finite time \cite{foiasTemamJFA}; see also \cite{biswasLocal}.
Ferrari and Titi later study general parabolic evolution equations on periodic domains,
and state that the linear lower bound for the radius of analyticity of solutions can be proved as long as the Sobolev norm
of the solution remains bounded \cite{ferrariTiti}.  Biswas and Swanson also study the three-dimensional
Navier-Stokes equations in the periodic
case, and demonstrate with data in weighted $\ell^{p}$ spaces that one finds small solutions existing for all time, with radius of
analyticity growing linearly in time \cite{biswasSwanson}.
In contrast with these results, in the present work
we are using the Wiener-based norms of Lei and Lin, and we also include
the improved lower bound on the radius of analyticity at small times.

It will be of interest to explore the actual behavior of the radius of analyticity of such solutions
in more detail in the future, including possibly by computational methods.  The proof of our lower
bound, which informally is like $\max\{\sqrt{t},t\}$
(see Theorem \ref{analyticityNS} in Section \ref{radiusSection} for the technical statement)
would carry over to other nonlinear parabolic
equations such as Burgers’ equation.  It will especially be of interest to determine whether there
is indeed a sharp change in behavior of the radius of analyticity, switching from a fractional power
to linear growth.  Some limited studies have been made before on the radius of analyticity of solutions of Burgers’ equation, such as
\cite{senouf}, but more detailed studies seem to be called for.

The plan of this paper is as follows.  In Section \ref{existenceSection} we state and prove
Theorem \ref{well-posedNS}, on existence, uniqueness, and continuous dependence of the
Lei-Lin solution of the three-dimensional Navier-Stokes equations on the torus. We present a detailed argument, following the ideas outlined by Bae in \cite{baePAMS} for the full-space case.  In Section \ref{radiusSection} we state and prove Theorem \ref{analyticityNS},
developing both the fractional-power and linear bounds for the radius of analyticity of the solutions.

%

\section{Well-posedness}\label{existenceSection}

We will make frequent use of the Fourier coefficients of a periodic function, so we will now be definite about what these
Fourier coefficients are.  We let the three-dimensional torus $\mathbb{T}^{3}$ be the cube $[0,2\pi]^{3},$ with periodic boundary
conditions.
Given a smooth periodic function, $f,$ we may write
\begin{equation}\nonumber f(x)=\sum_{k\in\mathbb{Z}^{3}}\hat{f}(k)e^{ik\cdot x}.
\end{equation}
Alternatively, when convenient we may use the notation $\mathcal{F}f(k)=\hat{f}(k).$  The Fourier coefficients are given
by the formula
\begin{equation}\nonumber
\hat{f}(k)=\frac{1}{(2\pi)^{3}}\int_{\mathbb{T}^{3}}f(x)e^{ik\cdot x}\ \mathrm{d}x.
\end{equation}

Recall the following function space introduced in \cite{LeiLin2011}, adapted to the periodic setting:
\[X^{-1} \equiv \left\{ f \in \mathcal{D}^\prime(\torus^3) \,|\, \hat{f}(0)=0 \text{ {\rm and} } \|f\|_{X^{-1}} \equiv \sum_{k\in \Z^3_\ast}  \frac{|\hat{f}(k)|}{|k|} < \infty\right\}.\]
Here we have used the notation $\Z^3_\ast \equiv \Z^3 \setminus (0,0,0)$.
As observed in \cite{LeiLin2011}, the space $X^{-1}$ is contained in $BMO^{-1}$, the largest space in which well-posedness for 3D Navier-Stokes with small initial data has been established, see \cite{KochTataru2001}.
%
%
In \cite{baePAMS} two additional function spaces were introduced in order to establish existence of a solution to the incompressible Navier-Stokes equations in $\R^3$. We adapt the definitions of those function spaces to the periodic setting, while retaining the notation from \cite{baePAMS}:
%
%
\begin{multline}\nonumber
\mathcal{X}^{-1} \equiv \Bigg\{ f \in \mathcal{D}^\prime(\R_+\times\torus^3) \,|\,
\hat{f}(\cdot,0)=0 \text{ {\rm in} }\mathcal{D}^\prime(\R_+) \text{ {\rm and} } \\
\|f\|_{\mathcal{X}^{-1}} \equiv \sum_{k \in \Z^3_\ast} \sup_{t\geq 0} \frac{|\hat{f}(t,k)|}{|k|} < \infty\Bigg\},
\end{multline}
%
\[\mathcal{X}^{1} \equiv \left\{ f \in \mathcal{D}^\prime(\R_+\times\torus^3) \,|\, \|f\|_{\mathcal{X}^{1}} \equiv \sum_{k\in \Z^3_\ast} \int_0^{\infty}  |k||\hat{f}(t,k)|
\dd t < \infty\right\}.\]

In \cite{Jlali2016} L. Jlali observed that, when the physical space is $\real^3$ instead of $\torus^3$, there is a difficulty with the Fourier reconstruction which prevents completeness of a space which is similar to $\mathcal{X}^{1}$. We note that no such issue arises in our case since we deal with periodic functions. In fact, it is easy to see that $\mathcal{X}^{-1}$ and $\mathcal{X}^{1}$ are, both, Banach spaces with norms $\|\cdot\|_{\mathcal{X}^{-1}}$ and $\|\cdot\|_{\mathcal{X}^{1}}$, respectively. We will need to work in the intersection of these spaces, $\mathcal{X}^{-1} \cap \mathcal {X}^1$. To simplify notation we introduce the norm $\vertiii{\cdot}$ on $\mathcal{X}^{-1} \cap \mathcal {X}^1$:
\[ \vertiii{u} \equiv \|u\|_{\mathcal{X}^{-1}} + \|u\|_{\mathcal{X}^{1}}.\]

We consider the initial-value problem for the incompressible Navier-Stokes equations on $\torus^3$ with viscosity $\mu > 0$, given by the system:
\begin{equation}\label{NStorus}
\left\{
\begin{array}{ll}
  \partial_t v + (v \cdot \nabla) v = -\nabla p + \mu \Delta v & \text{ in } (0,\infty) \times \torus^3,\\
  \dv v = 0 & \text{ in } [0,\infty) \times \torus^3, \\
  v(0,\cdot) = v_0 & \text{ on } \{t=0\} \times \torus^3.
\end{array}
\right.
\end{equation}

In this section we will prove the existence of a {\em mild solution} $v \in \mathcal{X}^{-1} \cap \mathcal {X}^1$ to \eqref{NStorus} with initial data $v_0 \in X^{-1}$, as long as $v_0$ is sufficiently small. We begin with the definition of a mild solution of \eqref{NStorus} in $\mathcal{X}^{-1} \cap \mathcal {X}^1$. We use the notation $\PP$ for the Leray projector, $\PP = \mathbb{I} - \nabla \Delta^{-1} \dv$.

\begin{definition} \label{mildsolution}
Let $v_0 \in X^{-1}$.  We say that $v \in \mathcal{X}^{-1} \cap \mathcal {X}^1$ is a {\em mild solution} of \eqref{NStorus} with initial data $v_0$ if
\begin{equation} \label{mildformula}
v(t,\cdot) = e^{\mu t \Delta}[v_0] - \int_0^t e^{\mu (t-s) \Delta} \left[\PP \dv (v \otimes v) (s,\cdot)\right] \dd s.
\end{equation}
\end{definition}

To simplify notation in what follows we will write, throughout, $\mathcal{Y} \equiv \mathcal{X}^{-1} \cap \mathcal {X}^1$. As previously noted, $\mathcal{Y}$ is a Banach space with the norm $\vertiii{\cdot}$.

We are now ready to state our well-posedness result for \eqref{NStorus} in $\mathcal{Y}$ for sufficiently small initial data $v_0 \in X^{-1}$.

\begin{theorem} \label{well-posedNS}
  Let $v_0 \in X^{-1}$. Then there exists $\varepsilon_0$ such that, if $\| v_0 \|_{X^{-1}} < \varepsilon_0$, then there exists one and only one mild solution of the incompressible Navier-Stokes equations \eqref{NStorus} with initial data $v_0$ and such that
  \[\vertiii{v} \lesssim \|v_0\|_{X^{-1}}.\]
  Furthermore, the solution $v$ depends continuously on the initial data $v_0$.
\end{theorem}

The proof of Theorem \ref{well-posedNS} relies on a key estimate in $\mathcal{Y}$. We state this estimate in the lemma below.

\begin{lemma} \label{keyEstimateBilinearTerm}
Let $F$, $G \in \mathcal{Y}$. Consider the self-adjoint bilinear operator $B:\mathcal{Y}\times\mathcal{Y} \to \mathcal{Y}$ given by
\begin{equation} \label{bilinop}
B(F,G) \equiv \int_0^t e^{\mu (t-s) \Delta} \left[\PP \dv (F \otimes G) (s,\cdot)\right]\dd s.
\end{equation}

Then
\begin{equation}\label{keyestimate}
\vertiii{B(F,G)} \leq C\left(1+\frac{1}{\mu}\right)\vertiii{F}\vertiii{G}.
\end{equation}
\end{lemma}

\begin{proof}

The proof of this lemma involves calculations which are very similar to those in \cite[(2.2) to (2.10)]{baePAMS}, adapted to the periodic case.

We must show that both $\|B(F,G)\|_{\mathcal{X}^{-1}}$ and $\|B(F,G)\|_{\mathcal{X}^{1}}$ are bounded by the right-hand-side of \eqref{keyestimate}.

Let us begin with $\|B(F,G)\|_{\mathcal{X}^{-1}}$.
Consider the $k$-th Fourier coefficient, $k \in \Z^3_\ast$, of $B(F,G)$:
  \begin{multline} \label{kthcoeff}
  \hat{B(F,G)}(t,k) = \int_0^t e^{-\mu(t-s)|k|^2} \sum_{i,j=1}^3\left( 1 - \frac{k}{|k|^2}k_i \right) k_j \widehat{F_i G_j}(s,k) \dd s \\
  = \int_0^t e^{-\mu(t-s)|k|^2} \sum_{i,j=1}^3\left( 1 - \frac{k}{|k|^2}k_i \right) k_j \sum_{\ell \in \Z^3} \hat{F_i}(s,\ell) \hat{G_j}(s,k-\ell) \dd s .
  \end{multline}

Therefore we have
\begin{equation}\label{relevest}
\left| \hat{B(F,G)}(t,k)\right|
 \lesssim \int_0^t |k|e^{-\mu(t-s)|k|^2}\sum_{i,j=1}^3 \sum_{\ell \in \Z^3} |\hat{F_i}(s,\ell) ||\hat{G_j}(s,k-\ell)| \dd s .
\end{equation}

  Since $k\in \Z^{3}_\ast,$ we may multiply the right-hand-side of \eqref{relevest} by $1/|k|.$ We will also use the bound
  \begin{equation}\nonumber
    1 \leq \frac{|k-\ell|}{|\ell|} + \frac{|\ell|}{|k-\ell|},
  \end{equation}
for every $\ell \in \Z^3_\ast$, $\ell \neq k$. We then obtain that:
\begin{multline} \label{X-1keyestBeg}
\frac{1}{|k|} \left| \hat{B(F,G)}\,(t,k)\right| 
\lesssim \frac{1}{|k|} \int_0^t |k|e^{-\mu(t-s)|k|^2}\sum_{i,j=1}^3 \sum_{\ell \in \Z^3} |\hat{F_i}(s,\ell)|| \hat{G_j}(s,k-\ell) |\dd s  \\
  \leq C\int_0^t \sum_{i,j=1}^3 \sum_{\ell \in \Z^3} |\hat{F_i}(s,\ell) ||\hat{G_j}(s,k-\ell)| \dd s\\
  \leq C\sum_{i,j=1}^3 \sum_{\ell \in \Z^3_\ast,\ell \neq k} \int_0^t\left( |k-\ell| |\hat{G_j}(s,k-\ell)| \frac{|\hat{F_i}(s,\ell) |}{|\ell|} \right. 
  + \left. |\ell||\hat{F_i}(s,\ell)| \frac{|\hat{G_j}(s,k-\ell)|}{|k-\ell|} \right) \dd s \\
    \leq C\sum_{i,j=1}^3 \sum_{\ell \in \Z^3_\ast, \ell \neq k} \left[
    \sup_{0\leq s \leq t} \frac{|\hat{F_i}(s,\ell) |}{|\ell|} \int_0^t |k-\ell| |\hat{G_j}(s,k-\ell)| \dd s \right. \\
\\ \left.
    + \sup_{0\leq s \leq t}  \frac{|\hat{G_j}(s,k-\ell)|}{|k-\ell|}
    \int_0^t
    |\ell||\hat{F_i}(s,\ell)| \dd s   \right].
\end{multline}

We now take the $\sup_{t\geq 0}$ followed by the $\sum_{k \in \Z^3_\ast}$ in \eqref{X-1keyestBeg} to find:
\begin{multline}\nonumber
 \|B(F,G)\|_{\mathcal{X}^{-1}}
  = \sum_{k \in \Z^3_\ast} \sup_{t\geq 0} \frac{1}{|k|} \left| \hat{B(F,G)}\,(t,k)\right|   \\
   \leq C\sum_{k \in \Z^3_\ast} \sum_{i,j=1}^3 \sum_{\ell \in \Z^3_\ast, \ell \neq k} \left[
    \sup_{s \geq 0} \frac{|\hat{F_i}(s,\ell) |}{|\ell|} \int_0^{\infty} |k-\ell| |\hat{G_j}(s,k-\ell)| \dd s \right. \\
\\ \left.
    + \sup_{s\geq 0}  \frac{|\hat{G_j}(s,k-\ell)|}{|k-\ell|}
    \int_0^{\infty}|\ell||\hat{F_i}(s,\ell)| \dd s
        \right] \\
  = C\sum_{i,j=1}^3 \left[\left(\sum_{\ell \in \Z^3_\ast}\sup_{s \geq 0} \frac{|\hat{F_i}(s,\ell) |}{|\ell|}\right)
  \left(\sum_{m \in \Z^3_\ast} \int_0^{\infty} |m| |\hat{G_j}(s,m)| \dd s\right) \right. \\
  \left. +  \left(\sum_{\ell \in \Z^3_\ast} \int_0^{\infty}|\ell||\hat{F_i}(s,\ell)| \dd s \right)
  \left(\sum_{m \in \Z^3_\ast}\sup_{s\geq 0}  \frac{|\hat{G_j}(s,m)|}{|m|}\right)
   \right]\\
   \\
   \leq C\| F \|_{\mathcal{X}^{-1}}\| G \|_{\mathcal{X}^{1}} + C\| G \|_{\mathcal{X}^{-1}}\| F \|_{\mathcal{X}^{1}}
   \leq C\vertiii{F}\vertiii{G}.
\end{multline}

Note that, in the calculation above, we made use of the change of variables $m=k-\ell$.

Next, we will estimate $\|B(F,G)\|_{\mathcal{X}^{1}}$. We multiply the $k$-th Fourier coefficient of $B(F,G)$,
$k \in \Z^3_\ast$, by $|k|$ and we use \eqref{kthcoeff} and \eqref{relevest} to find:
\begin{equation} \label{X1keyestBeg}
 |k| \left|\hat{B(F,G)}\,(t,k)\right| 
  \leq C\int_0^t |k|^2 e^{-\mu(t-s)|k|^2} \sum_{i,j=1}^3 \sum_{\ell \in \Z^3} |\hat{F_i}(s,\ell) ||\hat{G_j}(s,k-\ell)| \dd s.
\end{equation}

Let us integrate \eqref{X1keyestBeg} in time over $\R_+$ and exchange the order of integration. With this we obtain:
\begin{multline}\nonumber
\int_0^{\infty}  |k| \left|\hat{B(F,G)}\,(t,k)\right| \dd t \\
  \leq C\int_0^{\infty} \int_s^{\infty} |k|^2 e^{-\mu(t-s)|k|^2} \sum_{i,j=1}^3 \sum_{\ell \in \Z^3} |\hat{F_i}(s,\ell) ||\hat{G_j}(s,k-\ell)| \dd t \dd s \\
  = \frac{C}{\mu} \int_0^{\infty} \sum_{i,j=1}^3 \sum_{\ell \in \Z^3} |\hat{F_i}(s,\ell) ||\hat{G_j}(s,k-\ell)| \dd s.
\end{multline}

Finally, summing in $k \in \Z^3$ we deduce:
\begin{multline}\nonumber
 \left\|B(F,G) \right\|_{\mathcal{X}^{1}}
= \sum_{k \in \Z^3} \int_0^{\infty}  |k| \left|\hat{B(F,G)}\,(t,k)\right| \dd t \\
  \leq \sum_{k \in \Z^3} \frac{C}{\mu} \int_0^{\infty} \sum_{i,j=1}^3 \sum_{\ell \in \Z^3} |\hat{F_i}(s,\ell) ||\hat{G_j}(s,k-\ell)| \dd s \\
  \leq \sum_{k \in \Z^3} \frac{C}{\mu} \sum_{i,j=1}^3 \sum_{\ell \in \Z^3_\ast,\ell \neq k} \int_0^{\infty}\left( |k-\ell| |\hat{G_j}(s,k-\ell)| \frac{|\hat{F_i}(s,\ell) |}{|\ell|} \right. \\
  + \left. |\ell||\hat{F_i}(s,\ell)| \frac{|\hat{G_j}(s,k-\ell)|}{|k-\ell|} \right) \dd s \\
        \leq \frac{C}{\mu}\, \sum_{k \in \Z^3} \sum_{i,j=1}^3 \sum_{\ell \in \Z^3_\ast, \ell \neq k} \left[
    \sup_{s \geq 0} \frac{|\hat{F_i}(s,\ell) |}{|\ell|} \int_0^{\infty} |k-\ell| |\hat{G_j}(s,k-\ell)| \dd s \right. \\
\\ \left.
    + \sup_{s \geq 0}  \frac{|\hat{G_j}(s,k-\ell)|}{|k-\ell|}
    \int_0^{\infty}   |\ell||\hat{F_i}(s,\ell)| \dd s  \right]\\
    \leq \frac{C}{\mu}\,\| F \|_{\mathcal{X}^{-1}}\| G \|_{\mathcal{X}^{1}} + \frac{C}{\mu}\,\| G \|_{\mathcal{X}^{-1}}\| F \|_{\mathcal{X}^{1}}
   \leq \frac{C}{\mu}\,\vertiii{F}\vertiii{G}.
\end{multline}

This concludes the proof of the lemma.
\end{proof}

The following classical abstract result is the standard tool to prove small data well-posedness results for mild solutions of Navier-Stokes.

\begin{lemma} \label{fixedpoint}
Let ($X,$ $\vertiii{\cdot}_{X}$) be a Banach space. Assume that $\mathcal{B}:X \times X \to X$ is a continuous bilinear operator and let $\eta>0$ satisfy $\eta\geq \|\mathcal{B}\|_{X\times X\rightarrow X}$. Then, for any $x_0 \in X$ such that
\[4\eta \vertiii{x_0}_{X}<1,\]
there exists one and only one solution to the equation
\[x=x_0+\mathcal{B}(x,x) \qquad \text{ with } \vertiii{x}_{X} < \frac{1}{2\eta}.\]
Moreover, $\vertiii{x}_{X} \leq 2\vertiii{x_0}_{X}$.
\end{lemma}

The hypotheses in Lemma \ref{fixedpoint} imply that the map $x \mapsto Tx \equiv x_0+B(x,x)$ is a contraction in
$\left\{z\in X \;|\; \|z\|< \frac{1}{2\eta}\right\}$ and, thus, has a fixed point. See also \cite[p. 37, Lemma 1.2.6]{Cannone1995} and \cite{AuscherTchamitchian1999,Cannone2003}.

\begin{proof}[Proof of Theorem \ref{well-posedNS}]



We begin by fixing $v_0 \in X^{-1}$.

According to Definition \ref{mildsolution} we seek $v \in \mathcal{Y}$ such that
\begin{equation} \label{fixedpointform}
v = e^{\mu t \Delta} [v_0] - B(v,v),
\end{equation}
where we are using the notation introduced in Lemma \ref{keyEstimateBilinearTerm}, \eqref{bilinop}.

We can use Lemma \ref{fixedpoint} with ($X = \mathcal{Y}$, $\vertiii{\cdot}$) and $x_0 = e^{\mu t \Delta} [v_0]$. According to our key estimate \eqref{keyestimate} in Lemma \ref{keyEstimateBilinearTerm}, we may take $\eta = \displaystyle{C\left(1+\frac{1}{\mu}\right)}$. Therefore, we have: if
\begin{equation} \label{indatacondsat}
4C\left(1+\frac{1}{\mu}\right) \vertiii{\,e^{\mu t \Delta\,} [v_0]}\|<1
\end{equation}
then there is one and only one $v \in \mathcal{Y}$, $\vertiii{v} \leq \mu/[2C(1+\mu)]$, which solves \eqref{fixedpointform}. Furthermore, the solution satisfies
\begin{equation} \label{soltnest}
\vertiii{v} \leq 2\vertiii{\,e^{\mu t \Delta} [v_0]\,}.
\end{equation}
The proof of Theorem \ref{well-posedNS} is concluded once we observe that
\[\hat{e^{\mu t \Delta} [v_0]}(t,k)= e^{-\mu t |k|^2}\hat{v_0}(k),\]
so that we easily obtain
\begin{equation}\label{hestimate}
  \vertiii{\,e^{\mu t \Delta} [v_0]\,} \leq \left(1+\frac{1}{\mu} \right)\|v_0\|_{X^{-1}}.
\end{equation}

Next we show the continuous dependence of the solution on initial data. To this end let us choose $u_0$ and $v_0 \in X^{-1}$ such that
\begin{equation}
\label{xtrahypindatau} 4C \left(1+\frac{1}{\mu}\right)^2 \|u_0\|_{X^{-1}} < 1  \quad \text{ and }
\quad
4C \left(1+\frac{1}{\mu}\right)^2 \|v_0\|_{X^{-1}} < 1.
\end{equation}
From \eqref{hestimate} we conclude that \eqref{indatacondsat} is satisfied for both $u_0$ and $v_0$. Let $u$, $v \in \mathcal{Y}$ be the unique solutions of \eqref{fixedpointform} with initial data $u_0$, $v_0$, respectively. Then we have
\begin{align} 
\nonumber
u - v & = e^{\mu t \Delta}[u_0] -e^{\mu t \Delta}[v_0] - B(u,u) + B(v,v)\\
\nonumber
      & = e^{\mu t \Delta}[u_0-v_0] - B(u-v,u+v),
\end{align}
where we have used that $B$ is self-adjoint and bilinear.

It follows from Lemma \ref{keyEstimateBilinearTerm} and from \eqref{hestimate} that
\begin{equation} \label{diffest}
\vertiii{u-v} \leq  \left(1+\frac{1}{\mu} \right)\|u_0 - v_0\|_{X^{-1}} + C\left(1+\frac{1}{\mu} \right)\vertiii{u-v} \left(\vertiii{u}+ \vertiii{v}\right).
\end{equation}
Putting together \eqref{soltnest} and \eqref{hestimate} we find
\begin{equation} \label{uandv}
\vertiii{u} + \vertiii{v} \leq 2 \left(1+\frac{1}{\mu} \right)(\|u_0\|_{X^{-1}} + \|v_0\|_{X^{-1}}).
\end{equation}
Therefore, substituting \eqref{uandv} in \eqref{diffest} leads to
\begin{equation}\nonumber
\vertiii{u-v} \leq  \left(1+\frac{1}{\mu} \right)\|u_0 - v_0\|_{X^{-1}} + 2C\left(1+\frac{1}{\mu} \right)^2 \left(\|u_0\|_{X^{-1}} + \|v_0\|_{X^{-1}}\right)\vertiii{u-v}.
\end{equation}
This then implies
\begin{equation} \label{continuity}
\left( 1 -  2C\left(1+\frac{1}{\mu} \right)^2 \left(\|u_0\|_{X^{-1}} + \|v_0\|_{X^{-1}}\right) \right) \vertiii{u-v} \leq  \left(1+\frac{1}{\mu} \right)\|u_0 - v_0\|_{X^{-1}}.
\end{equation}
The hypotheses we made on the initial data $u_0$, $v_0$, namely \eqref{xtrahypindatau}, 
imply that
\[ 1 -  2C\left(1+\frac{1}{\mu} \right)^2 \left(\|u_0\|_{X^{-1}} + \|v_0\|_{X^{-1}}\right) > 0.\]
Thus, continuity with respect to initial data follows from \eqref{continuity}. We note in passing that this is not (necessarily) uniform continuity.

\end{proof}

\begin{remark}
We note that the smallness condition on the initial data, with respect to the viscosity $\mu$, is:
\[\|v_0\|_{X^{-1}} < \varepsilon_0 < \frac{\mu^2}{4C(1+\mu)^2}.\]

This smallness condition contrasts with that obtained in \cite{LeiLin2011}, where it was only necessary to have $\|v_0\|_{X^{-1}}< \mu$.
\end{remark}

\section{Bounds on the radius of analyticity of solutions}\label{radiusSection}


Our second result concerns {\em analyticity} of the solution obtained in Theorem \ref{well-posedNS}. We follow the strategy set forth
by Lemari\'{e}-Rieusset in \cite[Theorem 24.3]{Lemarie-Rieusset2002}, finding our refined estimate for the radius of analyticity
by using two different exponential weights.

\begin{theorem} \label{analyticityNS}
  Let $v_0 \in X^{-1}$. For any $\alpha \in (0,1),$ there exists $\varepsilon_0$ such that, if $\| v_0 \|_{X^{-1}} < \varepsilon_0$, then the solution found in Theorem \ref{well-posedNS} with initial data $v_0$ is analytic, with radius of analyticity $R_\alpha \geq \max\{\mu\sqrt{t},\alpha \mu t\}$, In fact it holds that
    \[\vertiii{e^{\mu\sqrt{t}|D|} v} + \vertiii{e^{\alpha\mu t |D|} v}\lesssim \|v_0\|_{X^{-1}},\]
    where $|D|$ is the operator whose Fourier multiplier is $\sum_{i=1}^3 |k_i|$, $k \in \Z^3$.
\end{theorem}

\begin{proof}

Fix $v_0\in X^{-1}$ and $v=v(t,x)$ the unique mild solution of \eqref{NStorus}, with initial data $v_0$, constructed in Theorem \ref{well-posedNS}. According to the definition of a mild solution, Definition \ref{mildsolution}, $v$ satisfies the identity in \eqref{mildformula}:
\[v(t,\cdot) = e^{\mu t \Delta}[v_0] - \int_0^t e^{\mu (t-s) \Delta} \left[\PP \dv (v \otimes v) (s,\cdot)\right] \dd s.\]

We will prove Theorem \ref{analyticityNS} in two steps, first by estimating $\vertiii{e^{\mu \sqrt{t}|D|} v}$ and then $\vertiii{e^{\alpha \mu t |D|} v}$. We follow the same ideas and calculations set forth for the $\sqrt{t}$ rate of gain of analyticity
in the proof of \cite[Theorem 1.3]{baePAMS}; we extend this to the linear growth rate of the radius of analyticity in Section
\ref{linearWeightSection}.

\subsection{Using $\mu \sqrt{t}$ in the exponential weight}

To bound $\vertiii{e^{\mu \sqrt{t} |D|} v}$ we introduce $V^1= e^{\mu \sqrt{t} |D|} v$, so that:
\begin{align} \label{mildformulaV1}
V^1(t,\cdot)& = \, e^{\mu \sqrt{t} |D|}e^{\mu t \Delta}[v_0]
- \int_0^t e^{\mu \sqrt{t} |D|}e^{\mu (t-s) \Delta} \left[\PP \dv (v \otimes v) (s,\cdot)\right] \dd s \nonumber \\
&= \, e^{\mu \sqrt{t} |D|}e^{\mu t \Delta}[v_0]  \nonumber \\
& - \int_0^t e^{\mu \sqrt{t} |D|}e^{\mu (t-s) \Delta}
\left[\PP \dv (e^{-\mu \sqrt{s}|D|} V^1 \otimes e^{-\mu \sqrt{s}|D|} V^1) (s,\cdot)\right] \dd s .
\end{align}

Let us introduce, for $F$, $G \in \mathcal{Y}$,
\begin{align}\label{tilB}
  \tilde{B}=&\tilde{B}(F,G) \nonumber \\
  \equiv &
  - \int_0^t e^{\mu \sqrt{t} |D|}e^{\mu (t-s) \Delta}
\left[\PP \dv (e^{-\mu \sqrt{s}|D|} F \otimes e^{-\mu \sqrt{s}|D|} G) (s,\cdot)\right] \dd s .
\end{align}

Then we will show that $V^1\in \mathcal{Y}$ is a solution of
\[V^1 = e^{\mu \sqrt{t} |D|}e^{\mu t \Delta}[v_0] + \tilde{B}(V^1,V^1). \]

We begin by observing that $e^{\mu \sqrt{t} |D|}e^{\mu t \Delta}[v_0]\in \mathcal{Y}$. To this end note that the $k$-th Fourier coefficient of $e^{\mu \sqrt{t} |D|}e^{\mu t \Delta}[v_0]$, $k \in \Z^3$, is
\[\mathcal{F}\left(e^{\mu \sqrt{t} |D|}e^{\mu t \Delta}[v_0]\right)(t,k) = e^{\mu \sqrt{t}|k|-\mu t|k|^{2}}\hat{v}_{0}(k).\]

We introduce the auxiliary function
\begin{equation}\label{auxfunct}
a(z)=z-\frac{1}{2}z^{2},
\end{equation}
and we note that for all $z\in\mathbb{R},$ we have $a(z)\leq\frac{1}{2}$.

Therefore we deduce the bound
\begin{equation}\label{L1Bound}
\left|\mathcal{F}\left(e^{\mu \sqrt{t} |D|}e^{\mu t \Delta}[v_0]\right)(t,k)\right| =e^{\mu a(\sqrt{t}|k|)}e^{-\frac{\mu}{2}t|k|^{2}}|\hat{v}_{0}(k)|\leq ce^{-\frac{\mu}{2}t|k|^{2}}|\hat{v}_{0}(k)|.
\end{equation}
In fact we can take $c=e^{\mu / 2}.$

It follows from \eqref{L1Bound}, in the same way as for \eqref{hestimate}, that
\begin{equation}\nonumber
  \vertiii{\,e^{\mu \sqrt{t} |D|} e^{\mu t \Delta} [v_0]\,} \leq c\left(1+\frac{2}{\mu} \right)\|v_0\|_{X^{-1}}.
\end{equation}

Next we treat the bilinear term $\tilde{B}$, defined through \eqref{tilB}. As in \eqref{kthcoeff}, we have
\begin{align*}
\hat{\tilde{B}(F,G)}(t,k)=& -\int_{0}^{t}e^{\mu\sqrt{t}|k|-\mu (t-s)|k|^{2}}
 \sum_{i,j=1}^3\left( 1 - \frac{k}{|k|^2}k_i \right) k_j && \\
 & \qquad \qquad \times \mathcal{F}\left((e^{-\mu \sqrt{s}|D|}F)_i (e^{-\mu \sqrt{s}|D|}G)_j\right)(s,k) \dd s &&\\
= &-\int_0^t e^{\mu\sqrt{t}|k|-\mu(t-s)|k|^2} \sum_{i,j=1}^3\left( 1 - \frac{k}{|k|^2}k_i \right) k_j && \\
 & \qquad  \times \sum_{\ell \in \Z^3} e^{-\mu \sqrt{s}|\ell|}\hat{F_i}(s,\ell) e^{-\mu \sqrt{s}|k-\ell|}\hat{G_j}(s,k-\ell) \dd s . &&
\end{align*}

Rearranging factors of exponentials we have
\begin{align*}
&\hat{\tilde{B}(F,G)}(t,k) & & \\
=&-\int_{0}^{t}e^{\mu(\sqrt{t}|k|-\sqrt{s}|k|-\frac{1}{2}(t-s)|k|^{2})}
e^{-\frac{\mu }{2}(t-s)|k|^{2}} \sum_{i,j=1}^3\left( 1 - \frac{k}{|k|^2}k_i \right) k_j&&\\
& \qquad \times \sum_{\ell \in \Z^3} e^{\mu\sqrt{s}(|k|-|\ell|-|k-\ell|)}\hat{F_i}(s,\ell) \hat{G_j}(s,k-\ell) \dd s . &&
\end{align*}

We may bound two of these exponential factors by constants.  First, note that since for all $k$, $\ell \in \Z^3$, the triangle inequality implies $|k|\leq|\ell|+|k-\ell|,$ we have
\begin{equation}\nonumber
e^{\mu\sqrt{s}(|k|-|\ell|-|k-\ell|)}\leq 1.
\end{equation}
Then, note that, using the auxiliary function introduced in \eqref{auxfunct}, we can rewrite the first exponential factor as follows:
\begin{equation}\nonumber
e^{\mu(\sqrt{t}|k|-\sqrt{s}|k|-\frac{1}{2}(t-s)|k|^{2})}=e^{\mu[a(\sqrt{t}|k|)-a(\sqrt{s}|k|)]}.
\end{equation}
If $s|k|^{2}$ is such that $a(\sqrt{s}|k|)\geq0,$ then $a(\sqrt{t}|k|)-a(\sqrt{s}|k|)\leq a(\sqrt{t}|k|)\leq\frac{1}{2}.$  If instead
$s|k|^{2}$ is such that $a(\sqrt{s}|k|)<0,$ then since $t\geq s$ we have $a(\sqrt{t}|k|)\leq a(\sqrt{s}|k|).$  Then
$a(\sqrt{t}|k|)-a(\sqrt{s}|k|)\leq 0.$  In either case, we have, for all $s,$ $t\geq 0$ and $k\in \Z^3,$
$a(\sqrt{t}|k|)-a(\sqrt{s}|k|)\leq\frac{1}{2}.$

With these considerations we now have
\begin{equation}\label{N1Bound}
|\hat{\tilde{B}(F,G)}(t,k)|\lesssim \int_{0}^{t}|k| e^{-\frac{\mu}{2}(t-s)|k|^{2}}\left[\sum_{i,j=1}^3 \sum_{\ell\in\mathbb{Z}^{3}}
|\hat{F}_i(s,k-\ell)||\hat{G}_j(s,\ell)|\right]\dd s.
\end{equation}
This is  the same estimate as in \eqref{relevest}, with $\mu/2$ in place of $\mu$. We then follow the remainder of the proof of Lemma \ref{keyEstimateBilinearTerm} to conclude that, indeed, $\tilde{B}$ is a bilinear operator from $\mathcal{Y}\times\mathcal{Y}$
into $\mathcal{Y}$ and
\[\vertiii{\tilde{B}(F,G)}\lesssim \vertiii{F}\vertiii{G}.\]

We can now apply the abstract result in Lemma \ref{fixedpoint} together with the estimate \eqref{L1Bound} to conclude that there exists one and only one solution $W$, such that $\vertiii{W} < (2\|\tilde{B}\|)^{-1}$, of equation \eqref{mildformulaV1}, also satisfied by $V^1$. Let $w\equiv e^{-\mu\sqrt{t}|D|}W$. Then it is easy to check that $w$ satisfies the same equation as $v$, namely \eqref{mildformula}, and $\vertiii{w}\leq\vertiii{W} < (2\|\tilde{B}\|)^{-1}$. Revisiting the statement of Lemma \ref{fixedpoint} we see that it is possible to substitute $\|\tilde{B}\|$ by $\|\tilde{B}\| + \|B\|$ by simply taking smaller initial data $v_0$. Then $\vertiii{w} \leq [2(\|\tilde{B}\|+\|B\|)]^{-1}< [2\|B\|]^{-1}$, so we conclude that $w=v$. Hence, $W=V^1$. Furthermore, we have
\[\vertiii{e^{\mu \sqrt{t}|D|}v}\equiv \vertiii{V^1}\lesssim \|v_0\|_{X^{-1}},\]
as desired.

\subsection{Using $\alpha t$ in the exponential weight}\label{linearWeightSection}
Let us fix $\alpha\in\left(0,1\right)$.  We now address the bound on $\vertiii{e^{\alpha \mu t |D|} v}$. To this end introduce
$V^2 = e^{\alpha \mu t |D|} v$, so that:
\begin{align}\nonumber
V^2(t,\cdot) = & \, e^{\alpha \mu t |D|}e^{\mu t \Delta}[v_0]
- \int_0^t e^{\alpha \mu t |D|} e^{\mu (t-s) \Delta} \left[\PP \dv (v \otimes v) (s,\cdot)\right] \dd s \nonumber \\
= & \, e^{\alpha \mu t |D|} e^{\mu t \Delta}[v_0]  \nonumber \\
& - \int_0^t e^{\alpha \mu t |D|} e^{\mu (t-s) \Delta}
\left[\PP \dv (e^{-\alpha \mu s|D|} V^2 \otimes e^{-\alpha \mu s|D|} V^2) (s,\cdot)\right] \dd s .\nonumber
\end{align}

Let us introduce, for $F$, $G \in \mathcal{Y}$,
\begin{align}\label{overlineB}
  \overline{B}=&\overline{B}(F,G) \nonumber \\
  \equiv &
  - \int_0^t e^{-\alpha \mu t|D|}e^{\mu (t-s) \Delta}
\left[\PP \dv (e^{-\alpha \mu s|D|} F \otimes e^{-\alpha \mu s|D|} G) (s,\cdot)\right] \dd s .
\end{align}

As before we will show that $V^2\in \mathcal{Y}$ is a solution of
\[V^2 = e^{\alpha \mu t |D|}e^{\mu t \Delta}[v_0] + \overline{B}(V^2,V^2). \]

We begin by bounding $e^{\alpha \mu t |D|}e^{\mu t \Delta}[v_0]$ in $\mathcal{Y}$.
The $k$-th Fourier coefficient of $e^{\alpha \mu t |D|}e^{\mu t \Delta}[v_0]$, $k \in \Z^3$, is
\begin{align}\nonumber
\mathcal{F}\left(e^{\alpha \mu t |D|}e^{\mu t \Delta}[v_0]\right)(t,k) = & e^{\alpha \mu t |k|}
e^{-\mu t|k|^{2}}\hat{v}_{0}(k) \nonumber \\
& =e^{\alpha \mu t|k|-\alpha\mu t|k|^{2}}e^{-(1-\alpha)\mu t|k|^{2}}\hat{v}_{0}(k).\nonumber
\end{align}
Then, in view of the constraint on $\alpha$, we note that the exponent of the first factor on the right-hand side satisfies
\begin{equation}\nonumber
\alpha \mu t|k|-\alpha\mu t|k|^{2}=\alpha\mu t|k|\left(1-|k|\right)\leq 0,
\end{equation}
for all $k\in\mathbb{Z}^{3}$ (we are using here the discreteness of the Fourier variable, which is not possible in the $\mathbb{R}^{3}$
setting of \cite{baePAMS}).

Therefore we deduce the bound
\begin{equation}\label{L2Bound}
\left|\mathcal{F}\left(e^{\alpha \mu t |D|}e^{\mu t \Delta}[v_0]\right)(t,k)\right|\leq e^{-(1-\alpha)\mu t|k|^{2}}|\hat{v}_{0}(k)|.
\end{equation}

It follows from \eqref{L2Bound}, in the same way as for \eqref{hestimate}, that
\begin{equation}\label{L2estimate}
  \vertiii{\,e^{\alpha \mu t |D|} e^{\mu t \Delta} [v_0]\,} \lesssim \left(1+\frac{1}{(1-\alpha)\mu} \right)\|v_0\|_{X^{-1}}.
\end{equation}

Next we treat the bilinear term $\overline{B}$, defined through \eqref{overlineB}. As in \eqref{kthcoeff}, we have
\begin{align*}
\hat{\overline{B}(F,G)}(t,k)=& -\int_{0}^{t}e^{\alpha \mu t |k|-\mu (t-s)|k|^{2}}
 \sum_{i,j=1}^3\left( 1 - \frac{k}{|k|^2}k_i \right) k_j && \\
 & \qquad \qquad \times
 \mathcal{F}{\left((e^{-\alpha \mu s|D|}F)_i (e^{-\alpha \mu s|D|}G)_j\right)}(s,k) \dd s &&\\
= &-\int_0^t e^{\alpha\mu t|k|-\mu(t-s)|k|^2} \sum_{i,j=1}^3 \left( 1 - \frac{k}{|k|^2}k_i \right) k_j && \\
 & \qquad  \times \sum_{\ell \in \Z^3} e^{-\alpha \mu s|\ell|}\hat{F_i}(s,\ell) e^{-\alpha \mu s|k-\ell|}\hat{G_j}(s,k-\ell) \dd s . &&
\end{align*}

Again we rearrange exponential factors, finding
\begin{align*}
&\hat{\overline{B}(F,G)}(t,k) \\
&=-\int_{0}^{t}e^{\alpha\mu  t|k|-\alpha \mu s|k|-\alpha\mu(t-s)|k|^{2}}e^{-(1-\alpha)\mu(t-s)|k|^{2}}\sum_{i,j=1}^3\left( 1 - \frac{k}{|k|^2}k_i \right) k_j \\
& \qquad  \qquad \times \sum_{\ell \in \Z^3} e^{\alpha \mu s(|k|-  |\ell|-  |k-\ell|)}\hat{F_i}(s,\ell) \hat{G_j}(s,k-\ell) \dd s .\\
\end{align*}

As before, we may bound two of the exponentials.  First, we again have $|k|\leq |\ell|+|k-\ell|,$ so
\begin{equation}\nonumber
e^{\alpha \mu s(|k|- |\ell|- |k-\ell|)}\leq 1.
\end{equation}
Next observe that
\begin{equation} \label{usedDiscretenesss}
\alpha\mu  t|k|-\alpha \mu s|k|-\alpha\mu(t-s)|k|^{2} = \alpha\mu (t-s)|k|\left(1 - |k|\right) \leq 0,
\end{equation}
for all $k\in\mathbb{Z}^{3},$
since $0\leq s \leq t.$
(Note that \eqref{usedDiscretenesss} once again uses discreteness of the Fourier variable.)
This implies
\begin{equation}\nonumber
e^{\alpha \mu t|k|-\alpha\mu  s|k|-\alpha\mu(t-s)|k|^{2}}\leq 1.
\end{equation}

With these considerations we now have
\begin{equation}\nonumber
|\hat{\overline{B}(F,G)}(t,k)|\lesssim \int_{0}^{t}|k| e^{-(1-\alpha)\mu(t-s)|k|^{2}}\left[\sum_{i,j=1}^3 \sum_{\ell\in\mathbb{Z}^{3}}
|\hat{F}_i(s,k-\ell)||\hat{G}_j(s,\ell)|\right]\dd s.
\end{equation}
This estimate is in the same spirit as \eqref{N1Bound}, from which we can conclude that $\overline{B}$ is a bilinear operator from $\mathcal{Y}\times\mathcal{Y}$ into $\mathcal{Y}$ and
\[\vertiii{\overline{B}(F,G)}\lesssim \vertiii{F}\vertiii{G}.\]

We can once more apply the abstract result in Lemma \ref{fixedpoint} together with the estimate \eqref{L2Bound} and use the same argument as we did for $V^1$ to conclude that
\[\vertiii{e^{\alpha \mu t |D|}v}\equiv \vertiii{V^2}\lesssim \|v_0\|_{X^{-1}},\]
as desired.

We now comment on how these bounds imply the claimed bound for the radius of analyticity.  For any $f\in\mathcal{X}^{-1},$
we have immediately that $f(t,\cdot)\in X^{-1}$ for any $t\geq0.$  Using the elementary fact that convergent series are bounded,
from $e^{\mu\sqrt{t}|D|}v(t,\cdot)\in X^{-1},$ we have the existence of $C>0$ such that
$|e^{\mu\sqrt{t}|D|}\hat{v}(t,k)|/|k|\leq C,$ for all $t$ and $k.$  Rearranging, this becomes
$|\hat{v}(t,k)|\leq C|k|e^{-\mu\sqrt{t}|k|},$
Then for all $b<\mu\sqrt{t},$ we have
$e^{b|k|}|\hat{v}(t,k)|\leq C|k|e^{(b-\mu\sqrt{t})|k|}.$
Since the right-hand side is in $\ell^{2},$ the left-hand side is also in $\ell^{2}.$
Then by the periodic analogue of Theorem IX.13 in \cite{reedSimon}, we conclude that $v$ is analytic with radius of analyticity
at least $\mu\sqrt{t}.$  We may repeat the argument on $e^{\alpha\mu t|D|}v$ as this is also in the space $\mathcal{X}^{-1},$
finding that $v$ is analytic with radius of analyticity at least $\alpha\mu t.$

This concludes the proof of Theorem \ref{analyticityNS}.

\end{proof}

We make two final remarks. First, that the maximum size of initial data, $\varepsilon_{0}$, guaranteed to exist from Theorem \ref{analyticityNS}, vanishes in the limit $\alpha\rightarrow1^{-}.$  So, the faster the linear rate at which one wishes to gain analyticity, the smaller the data one must take. This is clear from the dependence on $\alpha$ in the right-hand side of \eqref{L2estimate}, as this quantity is relevant in the abstract result Lemma \ref{fixedpoint}. Second, that the optimality of the lower bounds obtained, both in the present work and in its predecessors, is an interesting issue, which is left open.

\section*{Acknowledgments} The authors are grateful to Hantaek Bae for helpful conversations.
DMA gratefully acknowledges  National Science Foundation support through
grant DMS-1907684. MCLF was partially supported by CNPq, through grant \# 310441/2018-8, and FAPERJ, through  grant \# E-26/202.999/2017.
HJNL acknowledges the support of CNPq, through  grant \# 309648/2018-1, and of FAPERJ, through  grant \# E-26/202.897/2018.
In addition, MCLF and HJNL would like to thank the Isaac Newton Institute for Mathematical Sciences for support and hospitality during the program "Mathematical aspects of turbulence: where do we stand", when part of the work on this paper was undertaken. This work was
supported in part by: EPSRC Grant Number EP/R014604/1.

\bibliography{AL2_analyticity.bib}{}
\bibliographystyle{plain}

\end{document}